\documentclass[12pt]{amsart}

\usepackage{amsmath} 
\usepackage{amssymb} 
\usepackage{amsthm} 
\usepackage{dsfont} 
\usepackage{graphicx} 
\usepackage{enumerate} 

\newtheorem{theorem}{Theorem}[section]

\newtheorem{definition}[theorem]{Definition}
\newtheorem{example}[theorem]{Example}
\newtheorem{lemma}[theorem]{Lemma}

\begin{document}

\title[Alldifferent Kernel]{Characterization of the Alldifferent \\ Kernel by Hall Partitions \\ and a Calculation Method }

\author{Thomas Fischer}
\address{Frankfurt am Main, Germany.}

\email{dr.thomas.fischer@gmx.de}

\date{}     


\begin{abstract}
We consider a set-valued mapping between two finite sets and define the alldifferent kernel 
which describes the submapping of alldifferent selections. This submapping is characterized by
Hall partitions which are introduced in this paper. The existence of a Hall partition is equivalent 
to the Hall condition. The unicity of Hall partitions is proved and the unicity of an alldifferent 
selection is characterized. A calculation method for the determination of the Hall partition and
the alldifferent kernel is presented.
\end{abstract}




\maketitle

                                        %
                                        %
\section{Introduction} \label{S:intro}
We consider finite nonempty sets $X$ and $Y$, a set-valued mapping $F: X \longrightarrow 2^Y$ 
and we ask for the existence of an alldifferent selection for $F$. For the concept of set-valued 
mappings compare Berge \cite{Ber}. A selection for $F$ is a point-valued mapping 
$s: X \longrightarrow Y$ such that $s(x) \in F(x)$ for each $x \in X$.

\begin{definition} \label{D:alldifferent}
A point-valued mapping $s: X \longrightarrow Y$ is called alldifferent if $x_1 \ne x_2$ implies 
$s(x_1) \ne s(x_2)$ for each $x_1, x_2 \in X$.
\end{definition}

There exist various notations for alldifferent mappings. Sometimes they are called one-to-one 
mappings, a transversal, a choice function or a system of distinct representatives. 

The term alldifferent is widely used in the theory of constraint satisfaction problems. See 
Dechter and Rossi \cite{DR} for an overview on constraint satisfaction problems and see 
Hoeve \cite{Hoe} for an overview on alldifferent constraints. 

\begin{definition} \label{D:Hall}
A set-valued mapping $F$ satisfies the Hall condition if $\sharp F(W) \ge \sharp W$ for each 
subset $W \subset X$.
\end{definition}

The marriage theorem of P. Hall \cite{HalP} characterizes the existence of an alldifferent selection 
for $F$. His theorem states that $F$ admits an alldifferent selection if and only if $F$ satisfies the 
Hall condition.

Based on this result and ideas in the proof of Halmos and Vaughan \cite{HV} we define 
critical and non-reducible sets and the complement mapping in Section \ref{S:complement}. These terms 
are a necessary requirement in the definition of Hall partitions. Our main result in Section \ref{S:partition} 
is that there exists a Hall partition for $F$ if and only if $F$ satisfies the Hall condition. 

In Section \ref{S:kernel} we consider the alldifferent kernel $F^*$ of $F$, the alldifferent selections of 
$F$. We show, how Hall partitions describe $F^*$. In particular this result contains a new version 
of the marriage theorem.

We present a result, when $F$ satisfies $F^*=F$ (Section \ref{S:alldifferent}). 
We also show, that a Hall partition is uniquely determined (Section \ref{S:unicity1}) and characterize
in Section \ref{S:unicity2}, when there exists a unique alldifferent selection.

In Section \ref{S:algorithm} we define a calculation method which determines the alldifferent kernel of 
$F$. This method calculates a Hall partition of $F$ in finitely many steps if and only if the Hall 
condition is satisfied. All possible exits of the calculation method are analyzed. 

The results are illustrated with examples and are applied to Sudoku puzzles. 

Finally we collect some basic terms and notations. A partition of a set is the disjoint union of nonempty 
subsets. For a given set $W \subset X$ we define $F(W) = \cup_{x \in W} F(x)$. The restriction of $F$ 
on a subset $W \subset X$ is denoted by $F_{\mid W}$. A submapping of $F$ is another set-valued 
mapping $F^\prime: X \longrightarrow 2^Y$ with the property $F^\prime (x) \subset F(x)$ for each 
$x \in X$, where equality is allowed. Given a submapping $F^\prime$ and a selection $s$ of $F$ such
that $F^\prime(x) = \{s(x)\}$ we identify $s$ with the mapping $F^\prime$. In this sense $F^\prime$ 
is a point-valued mapping.

We (partially) order set-valued mappings by set inclusion, i.e., a set-valued mapping
$F: X \longrightarrow 2^Y$ is said to be larger than another set-valued mapping 
$G: X \longrightarrow 2^Y$ if $G$ is a submapping of $F$. A mapping $F$ in a 
set $\mathbf{G}$ of set-valued mappings is said to be the largest mapping in $\mathbf{G}$
if $F$ is larger than each $G \in \mathbf{G}$.

The symbol $\sharp$ denotes the number of elements (cardinality) of a finite set. The set of positive 
integers in denoted by $\mathbb{N}$. The union over an empty index set is considered to be empty. 
The sum over an empty index set is considered to be zero.

In all figures of this paper (except Fig. \ref{Fig5}) the $x$- and $y$-axis are scaled by the values 
in $X$ and $Y$ which are positive integers starting with $1$.

                                        %
                                        %
\section{The Complement Mapping} \label{S:complement}
We introduce critical and non-reducible sets, the complement mapping and show the relation
between these terms. These definitions will be later used in the definition of Hall partitions.

We introduce a complement mapping $F_{W,Z}: X \backslash W \longrightarrow 2^{Y \backslash Z}$
of $F$ depending on subsets $W \subset X$ and $Z \subset Y$ by $F_{W,Z} (x) = F(x) \backslash Z$ 
for each $x \in X \backslash W$. $F_{W,Z}$ is a submapping of $F_{\mid (X \backslash W)}$. 
In particular, $F_{\emptyset , \emptyset} = F$. It may happen, that $F_{W,Z}$ has empty images. 
We write $F_W$ instead of  $F_{W,F(W)}$. 

\begin{lemma} \label{L:21}
Let $W \subset X$. The following statements hold: 
\begin{enumerate}[(i)] 
\item $(X \backslash W) \backslash V = X \backslash (W \cup V)$ for each $V \subset X \backslash W$. 
\item $(F_W)_V = F_{W \cup V}$ for each $V \subset X \backslash W$.  
\item $F(X) = F(W) \cup F_W(X \backslash W)$.
\item $F(W) \cap F_W(X \backslash W) = \emptyset$.
\end{enumerate}
\end{lemma}
\begin{proof}
The equation ``(i)" is elementary set theory.  The remaining statements follow from the definition of $F_W$.
\end{proof} 

\begin{definition} \label{D:critical}
A subset $W \subset X$ is called a critical set of $F$ if $W$ is nonempty and $\sharp F(W) =\sharp W$.
\end{definition}

Critical sets had been considered by Easterfield \cite{Eas} and he called them ``exactly adjusted". The term 
critical had been introduced by M. Hall \cite{HalM}. Schrijver \cite{Sch} called them tight sets. The 
consideration of critical subsets had been used by Crook \cite{Cro} and Provan \cite{Pro} in their 
description of a strategy solving Sudoku puzzles. Crook used the term ``preemptive" set and Provan 
used the term ``pigeon-hole rule". 

\begin{example} \label{E:21}
On the left side in Fig. \ref{Fig1} the set $W = \{1, 2\}$ is a critical set of $F$, since 
$\sharp F(\{1, 2\}) = \sharp \{1, 2\}$. On the right side the set $W = \{3, 4\}$ is a critical set of $F$. 
In both cases $X \backslash W$ is not a critical set of $F$.
\end{example}

The next lemma states an argument which had been used by Halmos and Vaughan \cite{HV} in
their proof of the marriage theorem.

\begin{lemma} \label{L:23}
Let $F$ satisfy the Hall condition and let $W \subset X$ be a critical set of $F$. $F_W$ satisfies
the Hall condition.
\end{lemma}
\begin{proof} 
Let $W \subset X$ be a critical set of $F$ and let $V \subset X \backslash W$. By assumption
$\sharp F(W \cup V) \ge \sharp (W \cup V) = \sharp W + \sharp V$. By Lemma \ref{L:21},
$F(W \cup V) = F(W) \cup F_W(V)$ and $F(W) \cap F_W(V) = \emptyset$. This implies
\[
\sharp F_W(V) = \sharp F(W \cup V) - \sharp F(W) \ge \sharp W + \sharp V - \sharp W = \sharp V,
\]
i.e., $F_W$ satisfies the Hall condition.
\end{proof} 

\begin{figure}
\includegraphics{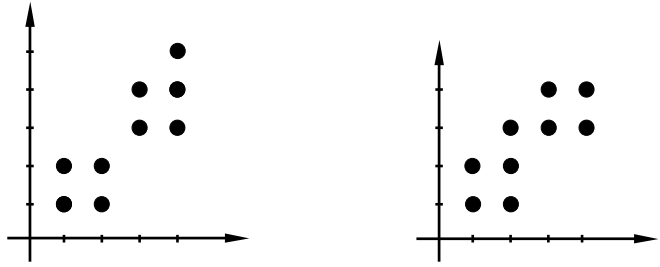}
\caption{Critical and non-reducible sets} 
\label{Fig1}
\end{figure}

\begin{definition} \label{D:non-reducible}
A subset $W \subset X$ is called a non-reducible set of $F$ if $W$ is nonempty and there does not exist a 
proper subset $W^\prime \subset W$ such that $W^\prime$ is a critical set of $F$.
\end{definition}

Each singleton is a non-reducible set. In the preceeding definition, we do not require, that $W$ is a 
critical set. Also the idea of non-reducible sets had been used by Halmos and Vaughan \cite{HV} in
their proof of the marriage theorem.

\begin{example} \label{E:22}
On the left in Fig. \ref{Fig1} the set $W=\{3,4\}$ is a non-reducible set of $F$, but $W$ is not critical. 
Define $W_1 = \{1,2\}$, then $W=\{3,4\}$ is not a critical 
set of $F_{W_1}$. On the right $W=\{1,2\}$ is also not a critical set of $F$. Define $W_1 = \{3,4\}$,
then $W=\{1,2\}$ is a critical and non-reducible set of $F_{W_1}$.
\end{example}

Each critical set contains a subset which is critical and non-reducible.

\begin{lemma} \label{L:25}
Let $W \subset X$ be a critical set of $F$. There exists a critical and non-reducible set $V \subset W$ of $F$. 
\end{lemma}
\begin{proof} 
Consider the nonempty collection of all critical sets $V \subset W$ of $F$. We choose a critical
set $V \subset W$ with the least number of elements, i.e., $\sharp V \le \sharp Z$ for each critical 
set $Z \subset W$. This $V$ is a critical and non-reducible set of $F$.
\end{proof} 

The non-reducible property implies the Hall condition.

\begin{lemma} \label{L:26}
Let $F$ have nonempty images and let $X$ be a non-reducible set of $F$.
\begin{enumerate}[(i)] 
\item $F$ satisfies the Hall condition.
\item $F_{\{x\},\{y\}}$ satisfies the Hall condition for each $x \in X$ and $y \in F(x)$. 
\end{enumerate}
\end{lemma}
\begin{proof} 
``(i)" Let $W \subset X$. We prove by induction on the number of elements $\sharp W$  of $W$ that
$\sharp F(W) \ge \sharp W$. This is true for $\sharp W = 1$, since $F$ has nonempty images.

Let $\sharp W > 1$. Choose some $x \in W$. By induction hypothesis
$\sharp F(W \backslash \{x\}) \ge \sharp (W \backslash \{x\})$ and this implies
$\sharp F(W \backslash \{x\}) \ge \sharp (W \backslash \{x\}) + 1$, since $W \backslash \{x\} \subset X$ is 
not a critical set of $F$. Consequently, 
$\sharp F(W) \ge \sharp F(W \backslash \{x\}) \ge \sharp (W \backslash \{x\}) + 1 = \sharp W$.

``(ii)"  Let $x \in X$, $y \in F(x)$ and $W \subset X \backslash \{x\}$. By ``(i)", 
$\sharp F(W) \ge \sharp W$. Since X is a non-reducible set of $F$, 
$\sharp F(W) \ge \sharp W + 1$.
This implies $\sharp F_{\{x\},\{y\}}(W) = \sharp (F(W) \backslash \{y\}) \ge \sharp F(W) - 1 \ge \sharp W$,
i.e., $F_{\{x\},\{y\}}$ satisfies the Hall condition.
\end{proof}

                                        %
                                        %
\section{The Hall Partition} \label{S:partition}
We consider the relation between partitions of $X$ and the complement mapping. Based on these 
properties we define Hall partitions.

\begin{lemma} \label{L:31}
Let $W_1, \ldots, W_m$, $m \in \mathbb{N}$, be a partition of $X$. 
\begin{enumerate}[(i)]
\item $F(W_i) \cap F_{\bigcup_{j=1}^{k-1} W_j}(W) = \emptyset$ for 
$W \subset X \backslash \bigcup_{j=1}^{k-1}W_j$ and $1 \le i < k \le m$,  
\item $F_{\bigcup_{j=1}^{i-1} W_j}(W_i) \cap F_{\bigcup_{j=1}^{i} W_j}(W) = \emptyset$ 
for $W \subset X \backslash \bigcup_{j=1}^{i}W_j$ and $i= 1, \ldots, m - 1$, and
\item $F_{\bigcup_{j=1}^{i-1} W_j}(W_i) \cap F_{\bigcup_{j=1}^{k-1} W_j}(W_k) = \emptyset$ 
for $i, k = 1, \ldots, m, i \ne k$.
\end{enumerate}
\end{lemma}
\begin{proof}
``(i)" Let $2 \le k \le m$ and $W \subset X \backslash \bigcup_{j=1}^{k-1}W_j$. By definition of the 
complement mapping $F_{\bigcup_{j=1}^{k-1}W_j}(W) = F(W) \backslash F(\cup_{j=1}^{k-1} W_j)$
and $F(W_i) \subset F(\cup_{j=1}^{k-1} W_j)$ for $1 \le i < k \le m$.  \\
``(ii)" Let $1 \le i \le m - 1$ and $W \subset X \backslash \bigcup_{j=1}^{i}W_j$. 
$F_{\bigcup_{j=1}^{i-1} W_j}(W_i) \subset F(W_i)$ and the statement follows from ``(i)"
with $k=i+1$.  \\
``(iii)" Let $1 \le i,k \le m$, $i \ne k$. W.l.o.g. $i < k$. 
$W_k \subset X \backslash \bigcup_{j=1}^{k-1}W_j$ and by definition of the complement mapping
$F_{\bigcup_{j=1}^{i-1} W_j}(W_i) \subset F(W_i)$. The statement follows from ``(i)". 
\end{proof} 

\begin{lemma} \label{L:32}
Let $W_1, \ldots, W_m$, $m \in \mathbb{N}$, be a partition of $X$, let $1 \le i \le m$ and 
$W \subset X \backslash \bigcup_{j=1}^{i-1}W_j$.
\begin{enumerate}[(i)]
\item $F(W \cup \bigcup_{j=1}^{i-1}W_j) 
= F_{\bigcup_{l=1}^{i-1} W_l}(W) \cup \bigcup_{j=1}^{i-1} F_{\bigcup_{l=1}^{j-1} W_l}(W_j)$ and  
\item $\sharp F(W \cup \bigcup_{j=1}^{i-1}W_j) 
= \sharp F_{\bigcup_{l=1}^{i-1} W_l}(W) + \sum_{j=1}^{i-1} \sharp F_{\bigcup_{l=1}^{j-1} W_l}(W_j)$.
\end{enumerate}
\end{lemma}
\begin{proof}
``(i)" Let $y \in F(W \cup \bigcup_{j=1}^{i-1}W_j) $. 
We distinguish two cases.  \\
Case 1: $y \in F(W_l)$ for some $1 \le l \le i-1$.  \\
Choose a minimal $1 \le i_0 \le i-1$ such that $y \in F(W_{i_0})$. This implies 
$y \notin F(\bigcup_{l=1}^{i_0 - 1}W_l)$ and there exists $x \in W_{i_0}$ such that $y \in F(x)$.  \\
Case 2: $y \notin F(W_l)$ for $1 \le l \le i-1$.  \\
Set $i_0=i$, i.e., $y \notin F(\bigcup_{l=1}^{i_0 - 1}W_l)$ and there exists $x \in W$ such that 
$y \in F(x)$.

In both cases we use the definition of the complement mapping
\[
F_{\bigcup_{l=1}^{i_0 - 1}W_l}(x) = F(x) \backslash F(\cup_{l=1}^{i_0 - 1}W_l),
\]
i.e., $y \in F_{\bigcup_{l=1}^{i_0 - 1}W_l}(x)$. In Case 1, 
$y \in \bigcup_{j=1}^{i-1} F_{\bigcup_{l=1}^{j-1} W_l}(W_j)$. In Case 2,
$y \in F_{\bigcup_{l=1}^{i-1} W_l}(W)$.
The ``$\supset$"-inclusion is obvious and this shows ``(i)".  \\  
``(ii)" Follows from ``(i)" and Lemma \ref{L:31} ``(ii)" and ``(iii)".
\end{proof} 

\begin{lemma} \label{L:34}
Let $W_1, \ldots, W_m$, $m \in \mathbb{N}$, be a partition of $X$. 
\begin{enumerate}[(i)]
\item $F(X) = \bigcup_{i=1}^{m}F_{\bigcup_{j=1}^{i-1} W_j}(W_i)$,  
\item $\sharp F(X) = \sum_{i=1}^{m} \sharp F_{\bigcup_{j=1}^{i-1} W_j}(W_i)$ and
\end{enumerate}
\end{lemma}
\begin{proof}
``(i)" Follows from Lemma \ref{L:32} ``(i)" with $i=m$ and $W=W_m$.   \\
``(ii)" Follows from Lemma \ref{L:32} ``(ii)" with $i=m$ and $W=W_m$. 
\end{proof} 

A Hall partition is a partition of $X$ with additional properties connected with the complement mapping.
As we will see later, a set-valued mapping admits a Hall partition if and only if it satisfies the Hall condition.

\begin{definition} \label{D:partition}
A tuple $(W_1, \ldots, W_m)$, $m \in \mathbb{N}$, is called a Hall partition of $F$ if 
$W_1, \ldots, W_m$ is a partition of $X$,
\begin{enumerate}[(i)] 
\item $F_{\bigcup_{j=1}^{i-1} W_j}$ has nonempty images on $W_i$ for $i=1, \ldots, m$,
\item $W_i$ is a non-reducible set of $F_{\bigcup_{j=1}^{i-1} W_j}$ for $i = 1, \ldots, m$ and 
\item $W_i$ is a critical set of $F_{\bigcup_{j=1}^{i-1} W_j}$ for $i = 1, \ldots, m-1$.
\end{enumerate}
\end{definition}

At first glance a Hall partition has nothing to do with the Hall condition, but Lemma \ref{L:26} and 
Theorems \ref{T:31} and \ref{T:32} justify the name. 

\begin{example} \label{E:31}
On the left in Fig. \ref{Fig2}, $W_1=\{1, 2, 3\}$ is a non-reducible set of $F$ and describes a Hall partition. 
In the middle the set $W_1 = \{1, 2\}$ describes a critical and non-reducible set of $F$ and 
$W_2 = \{3, 4\}$ describes a critical and non-reducible set of $F_{W_1}$. $(W_1,W_2)$ describes a 
Hall partition of $F$. On the right $W_1 = \{1, 2\}$ describes a critical and non-reducible set of $F$, 
$W_2 = \{3\}$ describes a critical and non-reducible set of $F_{W_1}$ and $W_3 = \{4\}$ describes 
a non-reducible set of $F_{W_1 \cup W_2}$ where $F_{W_1 \cup W_2}(W_3) = \{3,5\}$, i.e., $W_3$
is not a critical set of $F_{W_1 \cup W_2}$. $(W_1,W_2,W_3)$ describes a Hall partition of $F$.
\end{example}

\begin{figure}
\includegraphics{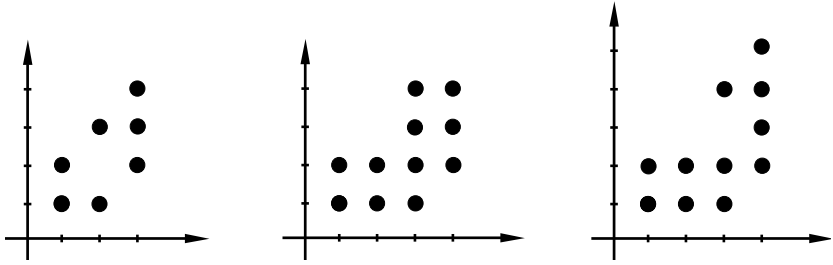}
\caption{Examples of Hall partitions} 
\label{Fig2}
\end{figure}

In contrast to partitions $W_1, \ldots, W_m$ of $X$, the ordering of a Hall partition $(W_1, \ldots, W_m)$
of $F$ is not arbitrary (compare Fig. \ref{Fig4} on the right). The examples in Fig. \ref{Fig2} suggest, 
that the ordering of a Hall partition is unique. But it is possible to construct examples where the 
ordering is not unique (compare Fig. \ref{Fig4} in the middle). It is also possible to construct examples
where any ordering of $(W_1, \ldots, W_m)$ describes a Hall partition of $F$ (compare Fig. \ref{Fig4} 
on the left).

We state additional properties of Hall partitions.

\begin{lemma} \label{L:35}
Let $(W_1, \ldots, W_m)$, $m \in \mathbb{N}$, be a Hall partition of $F$. 
\begin{enumerate}[(i)]
\item $\bigcup_{j=1}^{i-1}W_j$ 
is a critical set of $F$ for $i=2, \ldots, m$. and
\item $W_m$ is a critical set of $F_{\bigcup_{j=1}^{m-1} W_j}$ if and only if $\sharp F(X) = \sharp X$.
\end{enumerate}
\end{lemma}
\begin{proof}
``(i)" This is a consequence of Lemma \ref{L:32} ``(ii)" with $W = \emptyset$ and the definition of Hall
partitions.   \\
``(ii)"  Using Lemma \ref{L:34} ``(ii)" and the definition of Hall partitions, the equation
\[
\sharp F(X) = \sum_{i=1}^{m} \sharp F_{\cup_{j=1}^{i-1} W_j} (W_i)
= \sum_{i=1}^{m-1} \sharp W_i + \sharp W_m + \sharp F_{\cup_{j=1}^{m-1} W_j} (W_m) - \sharp W_m
\]
\[
= \sharp X + \sharp F_{\cup_{j=1}^{m-1} W_j} (W_m) - \sharp W_m 
\]
proves the statement.
\end{proof} 

\begin{theorem} \label{T:31}
Let $F$ satisfy the Hall condition. $F$ admits a Hall partition.
\end{theorem}
\begin{proof} 
We prove the statement by induction on the number of elements $\sharp X$ in $X$. Let $\sharp X = 1$. 
Choose $m = 1$ and $W_1 = X$. Then $F = F_\emptyset$ has nonempty images on $W_1$ and $W_1$ 
is a non-reducible set of $F = F_\emptyset$. This shows the statement for $\sharp X = 1$.

Let $\sharp X > 1$. By induction hypothesis the statement is true for sets $X^\prime$ with 
$\sharp X^\prime < \sharp X$. We distinguish several cases. \\
Case 1: $F$ admits a critical set. \\
By Lemma \ref{L:25} there exists a critical and non-reducible set $W_1 \subset X$ of $F$. We 
consider two subcases. \\
Case 1a: $W_1 = X$.  \\
$W_1$ is a critical and non-reducible set of $F$ and $(W_1)$ is a Hall partition of $F$.  \\
Case 1b: $W_1 \ne X$.  \\
$X \backslash W_1$ is nonempty and $\sharp (X \backslash W_1) < \sharp X$. By Lemma 
\ref{L:23}, $F_{W_1}$ satisfies the Hall condition. By induction hypothesis there exists a Hall 
partition $(W_2, \ldots, W_m)$, $m \in \mathbb{N}$, of $F_{W_1}$, i.e., $W_2, \ldots, W_m$
is a partition of $X \backslash W_1$,
\begin{enumerate}[(i)] 
\item $(F_{W_1})_{\bigcup_{j=2}^{i-1} W_j}$ has nonempty images on $W_i$ for $i=2, \ldots, m$,
\item $W_i$ is a non-reducible set of $(F_{W_1})_{\bigcup_{j=2}^{i-1} W_j}$ for $i = 2, \ldots, m$ and  
\item $W_i$ is a critical set of $(F_{W_1})_{\bigcup_{j=2}^{i-1} W_j}$ for $i = 2, \ldots, m-1$. 
\end{enumerate}
Using Lemma \ref{L:21} ``(ii)", $(F_{W_1})_{\bigcup_{j=2}^{i-1} W_j} = F_{\bigcup_{j=1}^{i-1} W_j}$
for $i = 2, \ldots, m$ and $(W_1, \ldots, W_m)$ describes a Hall partition of $F$. \\
Case 2: $F$ does not admit a critical set. \\
$W_1=X$ is a non-reducible set of $F$ and $(W_1)$ describes a Hall partition of $F$.
\end{proof} 

The converse of Theorem \ref{T:31} is also true.

\begin{theorem} \label{T:32}
Let $F$ admit a Hall partition. $F$ satisfies the Hall condition.
\end{theorem}
\begin{proof}
Let $(W_1, \ldots, W_m)$, $m \in \mathbb{N}$, be a Hall partition of $F$ and let $W \subset X$. 
Using Lemma \ref{L:31} ``(iii)" and \ref{L:26} applied to $F_{\bigcup_{j=1}^{i-1} W_j}$ for
$i=1, \ldots, m$,
\[
\sharp F(W) 
= \sharp \bigcup_{i=1}^{m}  F(W \cap W_i) 
\ge \sharp \bigcup_{i=1}^{m}  F_{\bigcup_{j=1}^{i-1} W_j}(W \cap W_i) 
\]
\[
= \sum_{i=1}^{m} \sharp F_{\bigcup_{j=1}^{i-1} W_j}(W \cap W_i) 
\ge \sum_{i=1}^{m} \sharp (W \cap W_i) = \sharp W,
\]
i.e., $F$ satisfies the Hall condition.
\end{proof}

                                        %
                                        %
\section{The Alldifferent Kernel} \label{S:kernel}
We define the alldifferent kernel $F^*$ of $F$ which describes the submapping with the alldifferent 
selections of $F$. Provided there exists a Hall partition we define a second submapping $F^p$. The 
relation between both submappings is exhibited. 

We define a submapping $F^*: X \longrightarrow 2^Y$ of $F$ by
\begin{align*}                                                   
F^*(x) = \{ y \in F(x) \mid 
& \mbox{ there exists an alldifferent selection } s \mbox{ of } F \\
& \mbox{ such that } y = s(x) \} 
\end{align*}
for each $x \in X$. We call $F^*$ the alldifferent kernel of $F$. 

A simple consideration shows $(F^*)^* = F^*$. A set-valued mapping $F$ admits an alldifferent 
selection if and only if $F^*(x) \ne \emptyset$ for all (or one) $x \in X$. The mapping $F^*$ 
describes the submapping of all alldifferent selections of $F$.

\begin{example} \label{E:41}
On the left in Fig. \ref{Fig1}, $F^* = F$. On the right there is no alldifferent selection of 
$F$ which passes through the point $x=2$, $y=3$, i.e., $F^* \ne F$.
\end{example}

We introduce another submapping of $F$ by means of a Hall partition. Let $(W_1, \ldots, W_m)$,
$m \in \mathbb{N}$, be a Hall partition of $F$. We define a submapping $F^p: X \longrightarrow 2^Y$ 
of $F$ by
\[
F^p_{\mid W_i} = ({F_{\bigcup_{j=1}^{i-1}W_j}})_{\mid W_i} \mbox{ for each }i=1, \ldots , m.
\]
In particular $F^p$ has nonempty images. Formally the definition of $F^p$ depends not only on $F$, 
but also on the Hall partition. We will see in Theorem \ref{T:41} that the definition of $F^p$ does not 
depend on the choice of the Hall partition.

\begin{lemma} \label{L:44}
Let $F$ admit a Hall partition. $F^p$ is a submapping of $F^*$.
\end{lemma}
\begin{proof}
We prove the statement by induction on the number of elements $\sharp X$ of $X$.
The statement is true for $\sharp X = 1$. Let $\sharp X > 1$. Let $(W_1, \ldots, W_m)$,
$m \in \mathbb{N}$, be a Hall partition of $F$ and let $x \in X$, $y \in F^p (x)$. There exists 
$1 \le i_0 \le m$ such that $x \in W_{i_0}$. Define $x_{i_0} = x$ and $y_{i_0} = y$. Choose 
$x_i \in W_i$ and $y_i \in F^p(x_i)$ for $i=1, \ldots, m$, $i \ne i_0$. 
$(F_{\bigcup_{j=1}^{i-1}W_j})_{\mid W_i}$ admits a Hall partition $(W_i)$ and
$y_i \in F^p(x_i)=F_{\bigcup_{j=1}^{i-1}W_j}(x_i)$ for $i=1, \ldots, m$. By induction hypothesis 
there exist alldifferent selections $s_i$ of $(F_{\bigcup_{j=1}^{i-1}W_j})_{\mid W_i}$ such that
$s_i(x_i)=y_i$ for $i=1, \ldots, m$. Using Lemma \ref{L:31} ``(iii)", the $s_i$, $i=1, \ldots, m$, can 
be combined to an alldifferent selection $s$ of $F$ on $X$ such that $s(x_i)=y_i$ for $i=1, \ldots, m$. 
In particular, $y \in F^*(x)$.
\end{proof}

Lemma \ref{L:44} already shows that the existence of a Hall partition implies the existence of an 
alldifferent selection.

In their proof of the marriage theorem Halmos an Vaughan \cite{HV} distinguished two cases. In their 
Case 1 they assumed X to be a non-reducible set and in their Case 2 they assumed X not to be non-reducible. 
This case distinction can also been found in the present paper. Their Case 1 is contained in Lemma 
\ref{L:26}. Their Case 2 is contained in Case 1b of Theorem \ref{T:31} and Lemma \ref{L:44}. 

\begin{lemma} \label{L:45}
Let $W \subset X$ be a critical set of $F$ and let $s$ be an alldifferent selection of $F$. 
Then $s(x) \in F_W(x)$ for each $x \in X \backslash W$.
\end{lemma}
\begin{proof}
Let $x \in X \backslash W$. Since $s$ is a selection of $F$, $s(W) \subset F(W)$. Since $W$ is critical 
and $s$ is alldifferent $\sharp W = \sharp s(W) \le \sharp F(W) = \sharp W$. This implies $s(W) = F(W)$ 
and $s(x) \notin F(W)$, since $s$ is alldifferent. This shows $s(x) \in F(x) \backslash F(W) = F_W(x)$.
\end{proof}

\begin{lemma} \label{L:46}
Let $F$ admit a Hall partition $(W_1, \ldots, W_m)$, $m \in \mathbb{N}$, and let $s$ be an alldifferent 
selection of $F$. $s_{\mid \bigcup_{j=i}^{m}W_j}$ is an alldifferent selection of 
$F_{\bigcup_{j=1}^{i-1}W_j}$ for $i = 1, \ldots, m$.
\end{lemma}
\begin{proof}
We prove the statement by induction on $i=1, \ldots, m$. The statement is true for $i=1$. 
Let $2 \le i \le m$. By induction hypothesis $s_{\mid \bigcup_{j=i-1}^{m}W_j}$ is an alldifferent 
selection of $F_{\bigcup_{j=1}^{i-2}W_j}$. We apply Lemma \ref{L:45} where 
$\bigcup_{j=i-1}^{m}W_j \rightarrow X$, $W_{i-1} \rightarrow W$, 
$s_{\mid \bigcup_{j=i-1}^{m}W_j} \rightarrow s$ and $F_{\bigcup_{j=1}^{i-2}W_j}\rightarrow F$
and obtain that
$(s_{\mid \bigcup_{j=i-1}^{m}W_j})_{\mid \bigcup_{j=i}^{m}W_j} = s_{\mid \bigcup_{j=i}^{m}W_j}$ 
is an alldifferent selection of $(F_{\bigcup_{j=1}^{i-2}W_j})_{W_{i-1}} = F_{\bigcup_{j=1}^{i-1}W_j}$.
\end{proof}

The combination of Lemma \ref{L:44} and \ref{L:46} shows the identity of the submappings $F^*$ 
and $F^p$.

\begin{theorem} \label{T:41}
Let $F$ admit a Hall partition. $F^* = F^p$. 
\end{theorem}
\begin{proof}
The implication $F^p(x) \subset F^*(x)$ for each $x \in X$ had been proved in Lemma \ref{L:44}.
Let $x \in X$ and $y \in F^*(x)$.
By definition of $F^*$ there exists an alldifferent selection $s$ of $F$ such that $s(x)=y$. Let 
$(W_1, \ldots, W_m)$, $m \in \mathbb{N}$, be a Hall partition of $F$. There exists $1\le i \le m$ such 
that $x \in W_i$. By Lemma \ref{L:46}, $s_{\mid W_i}$ is a selection of
$(F_{\bigcup_{j=1}^{i-1}W_j})_{\mid W_i}$ and by definition of $F^p$, $y=s(x) \in F^p(x)$.
\end{proof}

We combine our results to a new version of the marriage theorem \cite{HalP}.

\begin{theorem}[Marriage Theorem] \label{T:42}
The following statements are \\ 
equivalent: 
\begin{enumerate}[(i)] 
\item $F$ satisfies the Hall condition.  
\item $F$ admits a Hall partition.  
\item $F$ admits an alldifferent selection.
\end{enumerate}
\end{theorem}
\begin{proof}
The implication ``(i) $\Rightarrow$ (ii)" is Theorem \ref{T:31} and the implication
``(ii) $\Rightarrow$ (iii)" is contained in Theorem \ref{T:41} (or Lemma \ref{L:44}). The implication
``(iii) $\Rightarrow$ (i)" is straightforward.
\end{proof}

                                        %
                                        %
\section{Alldifferent Mappings} \label{S:alldifferent}
We introduce the term alldifferent for set-valued mappings and collect properties of alldifferent mappings.

\begin{definition} \label{D:alldifferent2}
A set-valued mapping $F$ is called alldifferent if $F$ has nonempty images and $F^* = F$.
\end{definition}

If $F$ admits an alldifferent selection, $F^*$ has nonempty images and  $F^*$ is alldifferent.
A similar condition had been considered by Mohr and Masini \cite{MM} who called it ``arc consistent". 
Hoeve \cite{Hoe} called it ``hyper-arc consistent". 

\begin{lemma} \label{L:51}
Let $F$ admit an alldifferent selection. The mapping $F^*$ is the largest alldifferent submapping of $F$.
\end{lemma}
\begin{proof} 
By definition $F^*$ is a submapping of $F$. Let $F^\prime$ be an alldifferent submapping of $F$, i.e., 
$F^\prime$ has nonempty images, $F^\prime = (F^\prime)^*$ and $F^\prime (x) \subset F(x)$ 
for each $x \in X$. Let $x \in X$ and $y \in F^\prime (x)$. There exists an alldifferent selection $s$ 
of $F^\prime$ such that $s(x) = y$. Then $s$ is also a selection of $F$, i.e., $y \in F^* (x)$ and 
$F^\prime$ is a submapping of $F^*$.
\end{proof} 

\begin{lemma} \label{L:52}
The following statements are equivalent: 
\begin{enumerate}[(i)] 
\item $F$ is alldifferent. 
\item $F$ has nonempty images and $F_{\{x\},\{y\}}$ admits an alldifferent selection for each 
$x \in X$ and $y \in F(x)$. 
\end{enumerate}
\end{lemma}
\begin{proof} 
``(i) $\Rightarrow$ (ii)" Let $x \in X$ and $y \in F(x)$. Then $y \in F^*(x)$  and there exists an alldifferent 
selection $s$ of $F$ such that $y = s(x)$. The mapping $s_{\mid (X \backslash \{x\})}$ is an alldifferent 
selection of $F_{\{x\},\{y\}}$. \\
``(ii) $\Rightarrow$ (i)" By definition, $F^*$ is a submapping of $F$, i.e., $F^*(x) \subset F(x)$ for each 
$x \in X$. Let $x \in X$ and $y \in F(x)$. By assumption there exists an alldifferent selection $s^\prime$ of 
$F_{\{x\},\{y\}}$. Define
\[
s(z) =
\begin{cases}
y, & \mbox{ if }z = x, \\
s^\prime (z), & \mbox{ if }z \ne x
\end{cases}
\]
for each $z \in X$. The mapping $s$ is an alldifferent selection of $F$ and $s(x) = y$, i.e., $y \in F^*(x)$.
\end{proof} 

\begin{example} \label{E:51}
Define $X = Y = \{1, 2, 3, 4\}$ and consider the mapping whose graph is depicted in Fig. \ref{Fig3}.
The set $X$ is a critical and non-reducible set of $F$. The set-valued mapping $F_{\{4\},\{4\}}$ is 
alldifferent. The set-valued mapping $F_{\{1\},\{1\}}$ admits an alldifferent selection, but it is not 
alldifferent. The value $y=4$ is not contained in $(F_{\{1\},\{1\}})^*(3)$. 
\end{example}

\begin{figure}
\includegraphics{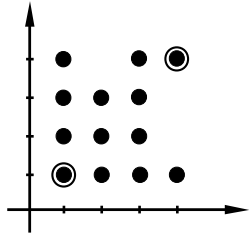}
\caption{An Alldifferent Mapping} 
\label{Fig3}
\end{figure}

\begin{lemma} \label{L:53}
Let $F$ be alldifferent. $F(W) \cap F(X \backslash W) = \emptyset$ for each critical 
set $W \subset X$ of $F$.
\end{lemma}
\begin{proof} 
Let $W \subset X$ be a critical set of $F$ and let $y \in F(X \backslash W)$. There exists 
$x \in X \backslash W$ such that $y \in F(x)$ and there exists an alldifferent selection $s$ of $F$ such that 
$s(x) = y$. Using Lemma \ref{L:45}, $y = s(x) \in F_W(x)=F(x) \backslash F(W)$, i.e., $y \notin F(W)$.
\end{proof} 

\begin{lemma} \label{L:54}
Let $(W_1, \ldots, W_m)$, $m \in \mathbb{N}$, be a Hall partition of $F$. Let 
$F(W) \cap F(X \backslash W) = \emptyset$ for each critical set $W \subset X$ of $F$. $F^p = F$
\end{lemma}
\begin{proof}
Let $2 \le i \le m$. Using Lemma \ref{L:35} ``(i)", $\bigcup_{j=1}^{i-1}W_j$ is a critical set of $F$.
This implies 
\[
F(W_i) \cap F(\cup_{j=1}^{i-1}W_j)  
\subset F(X \backslash \cup_{j=1}^{i-1}W_j) \cap F(\cup_{j=1}^{i-1}W_j) 
= \emptyset.
\]
Using the definition of the complement mapping 
$F^p_{\mid W_i}=(F_{\bigcup_{j=1}^{i-1}W_j})_{\mid W_i}=F_{\mid W_i}$ for $i=1, \ldots, m$.
\end{proof}

For alldifferent mappings the description of Hall partitions can be simplified.

\begin{theorem} \label{T:51}
The following statements are equivalent: 
\begin{enumerate}[(i)] 
\item $F$ is alldifferent.
\item $F$ admits an alldifferent selection and $F(W) \cap F(X \backslash W) = \emptyset$ for each 
critical set $W \subset X$ of $F$.
\item There exists a Hall partition $(W_1, \ldots, W_m)$, $m \in \mathbb{N}$, of $F$ such that $F^p = F$.
\end{enumerate}
\end{theorem}
\begin{proof} 
The implication ``(i) $\Rightarrow$ (ii)" is Lemma \ref{L:53}. The implication ``(ii) $\Rightarrow$ (iii)" 
is Theorem \ref{T:31} and Lemma \ref{L:54}. The implication ``(iii) $\Rightarrow$ (i)" is 
Theorem \ref{T:41}.
\end{proof}

                                        %
                                        %
\section{Unicity of Hall Partitions} \label{S:unicity1}
We show the integer $m \in \mathbb{N}$ and the sets $W_1, \ldots, W_m$ of a Hall partition are 
(up to renumbering) uniquely determined.

\begin{lemma} \label{L:unicity1.1}
Let $F$ satisfy the Hall condition and let $V$ and $W$ be critical sets of $F$. The following statements hold:
\begin{enumerate}[(i)] 
\item $\sharp (V \cap W) = \sharp F(V \cap W) = \sharp (F(V) \cap F(W))$.  
\item $\sharp (V \cup W) = \sharp F(V \cup W) = \sharp (F(V) \cup F(W))$.  
\end{enumerate}
\end{lemma}
\begin{proof}
Using the Hall condition and elementary set theory we obtain
\[
\sharp (V \cap W) \le \sharp F(V \cap W) \le \sharp (F(V) \cap F(W)) \mbox{ and}  
\]
\[
\sharp (V \cup W) \le \sharp F(V \cup W) \le \sharp (F(V) \cup F(W)).
\]
This implies
\begin{align*}
\sharp V + \sharp W 
& =\sharp (V \cap W) + \sharp (V \cup W)  \\
& \le \sharp (F(V) \cap F(W)) + \sharp (F(V) \cup F(W))  \\
& =\sharp F(V) + \sharp F(W)  \\
& =\sharp V + \sharp W
\end{align*}
and 
\[
\sharp (F(V) \cap F(W)) = \sharp (V \cap W) + \sharp (V \cup W) - \sharp (F(V) \cup F(W))
\le \sharp (V \cap W)
\]
showing equality in case ``(i)". The case ``(ii)" follows analoguously. 
\end{proof}

In particular the preceeding lemma shows that the intersection and union of two critical sets
is a critical set again. This result had been shown already by Everett and Whaples \cite[Lemma 1]{EW}.

\begin{lemma} \label{L:unicity1.2}
Let $W \subset X$ be a critical set of $F$. Let $(W_1, \ldots, W_m)$, $m \in \mathbb{N}$, be a Hall 
partition of $F$ and let $1 \le i \le m$ such that $W \subset X \backslash \bigcup^{i-1}_{j=1}W_j$. 
$W$ is a critical set of $F_{\bigcup^{i-1}_{j=1}W_j}$.
\end{lemma}
\begin{proof}
$\sharp F_{\bigcup^{i-1}_{j=1}W_j}(W) \le \sharp F(W) = \sharp W$, since $W$ is a critical set of $F$. 
By Lemma \ref{L:26}, $(F_{\bigcup^{l-1}_{j=1}W_j})_{\mid W_l}$ satisfies the Hall condition for 
$l=1, \ldots, m$. Using Lemma \ref{L:31} ``(iii)"
\[
\sharp F_{\bigcup^{i-1}_{j=1}W_j}(W)
= \sharp \bigcup^{m}_{l=i} F_{\bigcup^{i-1}_{j=1}W_j}(W \cap W_l)
\ge \sharp \bigcup^{m}_{l=i} F_{\bigcup^{l-1}_{j=1}W_j}(W \cap W_l)
\]
\[
= \sum^{m}_{l=i} \sharp F_{\bigcup^{l-1}_{j=1}W_j}(W \cap W_l)
\ge \sum^{m}_{l=i} \sharp (W \cap W_l)
= \sharp W.
\]
\end{proof}

\begin{lemma} \label{L:unicity1.3}
Let $(W_1, \ldots, W_m)$, $m \in \mathbb{N}$, be a Hall partition of $F$ and let $1 \le i \le m$ such 
that $W_i$ is a critical set of $F$. $F(W_i) \cap F(W_j) = \emptyset$ for $j=1, \ldots, i-1$.
\end{lemma}
\begin{proof}
Using Lemma \ref{L:unicity1.2}, $W_i$ is a critical set of $F_{\bigcup^{i-1}_{j=1}W_j}$, i.e.,
\[
\sharp W_i
= \sharp F_{\bigcup^{i-1}_{j=1}W_j}(W_i)
= \sharp (F(W_i) \backslash F(\cup^{i-1}_{j=1}W_j))
\le \sharp F(W_i)
= \sharp W_i
\]
and equality holds. This implies the statement.
\end{proof}

\begin{lemma} \label{L:unicity1.4}
Let $(W_1, \ldots, W_m)$, $m \in \mathbb{N}$, be a Hall partition of $F$ and let $W \subset X$ 
be a critical and non-reducible set of $F$. There exists $1 \le i \le m$ such that $W = W_i$.
\end{lemma}
\begin{proof}
There exists a minimal $1 \le i_0 \le m$ such that $W \cap W_{i_0} \ne \emptyset$. Using 
Lemma \ref{L:unicity1.2}, $W$ is a critical set of  $F_{\bigcup^{i_0-1}_{j=1}W_j}$. By definition of 
Hall partitions, $W_{i_0}$ is also a critical set of $F_{\bigcup^{i_0-1}_{j=1}W_j}$. By 
Lemma \ref{L:unicity1.1}, $W \cap W_{i_0}$ is a critical set of $F_{\bigcup^{i_0-1}_{j=1}W_j}$. 
Since $W$ and $W_{i_0}$ are non-reducible sets of $F_{\bigcup^{i_0-1}_{j=1}W_j}$, 
$W \cap W_{i_0} = W = W_{i_0}$.
\end{proof}

We introduce a general notation for Hall partitions which are created by eliminating sets $W_i$ from a 
given Hall partition $(W_1, \ldots, W_m)$, $m \in \mathbb{N}$, of $F$. Let $(W_1, \ldots, W_m)$,
$m \in \mathbb{N}$, be a Hall partition of $F$ and $\{i_1, \ldots, i_k\} \subset \{1, \ldots, m\}$, where 
$1 \le k < m$. The term $(W_1, \ldots, W_m) \backslash \{i_1, \ldots, i_k\}$ denotes a tuple
$(W^\prime_1, \ldots, W^\prime_{m-k})$ with $m-k$ components in $\{W_1, \ldots, W_m\}$ where 
sets with indices in $\{i_1, \ldots, i_k\}$ had been removed. The order of the remaining sets is preserved.

It is clear that $W^\prime_1, \ldots, W^\prime_{m-k}$ is a partition of 
$X \backslash \bigcup^{k}_{l=1}W_{i_l}$. We do not know if $(W^\prime_1, \ldots, W^\prime_{m-k})$
is still a Hall partition.

\begin{lemma} \label{L:unicity1.6}
Let $(W_1, \ldots, W_m)$, $m \in \mathbb{N}$, be a Hall partition of $F$ and let $1 \le i \le m$ such
that $W_i$ is a critical set of $F$. Let $(W^\prime_1, \ldots, W^\prime_{m-1})$ be the tuple 
$(W_1, \ldots, W_m) \backslash \{i\}$.
\[
((F_{W_i})_{\bigcup^{l-1}_{j=1}W^\prime_j})_{\mid W^\prime_l} =
\begin{cases}
(F_{\bigcup^{l-1}_{j=1}W_j})_{\mid W_l}      & \mbox{ for }l = 1, \ldots, i-1, \\
(F_{\bigcup^{l}_{j=1}W_j})_{\mid W_{l+1}}  & \mbox{ for }l = i, \ldots, m-1.
\end{cases}
\]
\end{lemma}
\begin{proof}
Let $1 \le l \le i-1$. Using Lemma \ref{L:unicity1.3}, $F(W_i) \cap F(W_l) = \emptyset$ and we obtain 
\[
((F_{W_i})_{\bigcup^{l-1}_{j=1}W^\prime_j})_{\mid W^\prime_l}
= ((F_{W_i})_{\bigcup^{l-1}_{j=1}W_j})_{\mid W_l}
= (F_{W_i \cup \bigcup^{l-1}_{j=1}W_j})_{\mid W_l}.
\]
This implies
\[
((F_{W_i})_{\bigcup^{l-1}_{j=1}W^\prime_j})(x)
= F(x) \backslash (F(W_i) \cup F(\cup^{l-1}_{j=1}W_j))
\]
\[
= F(x) \backslash F(\cup^{l-1}_{j=1}W_j)
= F_{\bigcup^{l-1}_{j=1}W_j}(x)
\]
for $x \in W_l =W^\prime_l$.

Let $i \le l \le m-1$. 
\[
((F_{W_i})_{\bigcup^{l-1}_{j=1}W^\prime_j})_{\mid W^\prime_l}
= (F_{W_i \cup \bigcup^{l-1}_{j=1}W^\prime_j})_{\mid W_{l+1}}
= (F_{\bigcup^{l}_{j=1}W_j})_{\mid W_{l+1}}.
\]
\end{proof}

\begin{lemma} \label{L:unicity1.8}
Let $(W_1, \ldots, W_m)$, $m \in \mathbb{N}$, be a Hall partition of $F$ and let $1 \le i \le m$
such that $W_i$ is a critical set of $F$. The tuple $(W^\prime_1, \ldots, W^\prime_{m-1})=
(W_1, \ldots, W_m) \backslash \{i\}$ is a Hall partition of $F_{W_i}$.
\end{lemma}
\begin{proof}
The statement follows from Lemma \ref{L:unicity1.6}, since $(W_1, \ldots, W_m)$ is a Hall partition of $F$.
\end{proof}

\begin{lemma} \label{L:unicity1.9}
Let $(W_1, \ldots, W_m)$, $m \in \mathbb{N}$, and $(W^\prime_1, \ldots, W^\prime_{m^\prime})$,
$m^\prime \in \mathbb{N}$, be Hall partitions of $F$ and let $1 \le k \le m^\prime -1$. There exists a subset
$\{i_1, \ldots, i_k\} \subset \{1, \ldots, m\}$ such that $W^\prime_j = W_{i_j}$ for $j=1, \ldots, k$ and
$(W_1, \ldots, W_m) \backslash$ $\{i_1, \ldots, i_k\}$ and 
$(W^\prime_{k+1}, \ldots, W^\prime_{m^\prime})$ are Hall partitions of 
$F_{\bigcup_{j=1}^{k}W_{i_j}}=F_{\bigcup_{j=1}^{k}W^\prime_j}$.
\end{lemma}
\begin{proof}
We prove the statement by induction on $k=1, \ldots, m^\prime -1$. The statement is true for $k=1$ 
by Lemma \ref{L:unicity1.4} and \ref{L:unicity1.8}. Let $2 \le k \le m^\prime -1$. By induction hypothesis 
there exists a subset $\{i_1, \ldots, i_{k-1}\} \subset \{1, \ldots, m\}$ such that $W^\prime_j = W_{i_j}$ 
for $j=1, \ldots, k-1$, $(W_1, \ldots, W_m) \backslash$ $\{i_1, \ldots, i_{k-1}\}$ and 
$(W^\prime_{k}, \ldots, W^\prime_{m^\prime})$ are Hall partitions of 
$F_{\bigcup_{j=1}^{k-1}W_{i_j}}=F_{\bigcup_{j=1}^{k-1}W^\prime_j}$. 

$W^\prime_k$ is a critical and non-reducible set of 
$F_{\bigcup^{k-1}_{j=1}W^\prime_j}=F_{\bigcup^{k-1}_{j=1}W_{i_j}}$.
By Lemma \ref{L:unicity1.4}, there exists $i_k \in \{1, \ldots, m\} \backslash \{i_1, \ldots, i_{k-1}\}$
such that $W^\prime_k=W_{i_k}$. By Lemma \ref{L:unicity1.8}, 
$(W_1, \ldots, W_m) \backslash \{i_1, \ldots, i_k\}$ is a Hall partition of $F_{\bigcup_{j=1}^{k}W_{i_j}}$
and $(W^\prime_{k+1}, \ldots, W^\prime_{m^\prime})$ is a Hall partition 
of $F_{\bigcup_{j=1}^{k}W^\prime_j}$.
\end{proof}

\begin{theorem} \label{T:unicity1.1}
Let $(W_1, \ldots, W_m)$, $m \in \mathbb{N}$, and $(W_1^\prime, \ldots, W^\prime_{m^\prime})$,
$m^\prime \in \mathbb{N}$, be Hall partitions of $F$. $m=m^\prime$ and there exists a 
renumbering $i_1, \ldots, i_m$ of $1, \ldots, m$ such that $W^\prime_j = W_{i_j}$ for $j=1, \ldots, m$.
\end{theorem}
\begin{proof}
We apply Lemma \ref{L:unicity1.9} with $k=m^\prime -1$. There exists a subset
$\{i_1, \ldots, i_{m^\prime -1}\} \subset \{1, \ldots, m\}$ such that 
$W^\prime_j = W_{i_j}$ for $j=1, \ldots, m^\prime -1$. This shows $m^\prime \le m$. 
Interchanging the role of $m$ and $m^\prime$ shows $m \le m^\prime$, i.e. $m=m^\prime$.

Let $i_{m^\prime} \in \{1, \ldots, m\} \backslash \{i_1, \ldots, i_{m^\prime -1}\}$. Then
\[
W^\prime_{m^\prime} 
= X \backslash \bigcup^{m^\prime -1}_{j=1}W^\prime_j
= X \backslash \bigcup^{m^\prime -1}_{j=1}W_{i_j}
= X \backslash \bigcup^{m}_{\begin{subarray}{c}j=1 \\ j \ne i_{m^\prime}\end{subarray}}W_{j}
= W_{i_{m^\prime}},
\]
i.e., $W^\prime_j = W_{i_j}$ for $j=1, \ldots, m^\prime$.
\end{proof}

The preceeding theorem states exactly the desired result. The Hall partition $(W_1, \ldots, W_m)$ 
of a set-valued mapping is (up to renumbering) uniquely determined.

If $(W_1, \ldots, W_m)$, $m \in \mathbb{N}$, is a Hall partition of $F$, the sets $W_1, \ldots, W_m$
are uniquely determined, but, in general, the ordering of $W_1, \ldots, W_m$ is not uniquely determined.

\begin{example} \label{E:unicity1.1}
Let $W_1=\{1,2\}$, $W_2=\{3,4\}$ and $W_3=\{5,6\}$ in Fig. \ref{Fig4}. In all three mappings 
$(W_1, W_2, W_3)$ is a Hall partition of $F$. On the left any other ordering of $(W_1, W_2, W_3)$ 
is also a Hall partition. In the middle $W_1$ is always a predecessor of $W_2$ and $W_3$, but the 
ordering between $W_2$ and $W_3$ is arbitrary. On the right the ordering $(W_1, W_2, W_3)$
is uniquely determined. A renumbering is not possible.
\end{example}

\begin{figure}
\includegraphics{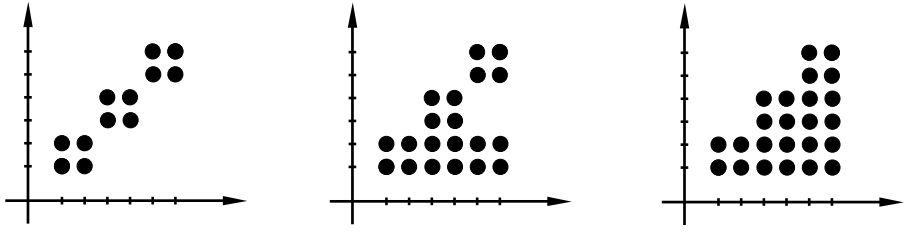}
\caption{Hall partitions} 
\label{Fig4}
\end{figure}

                                        %
                                        %
\section{Unicity of an Alldifferent Selection} \label{S:unicity2}
We show an equivalent condition, when a set-valued mapping admits a unique alldifferent selection.

\begin{lemma} \label{L:unicity2.2}
Let $(W_1, \ldots, W_m)$, $m \in \mathbb{N}$, be a Hall partition of $F$. The following statements 
are equivalent:
\begin{enumerate}[(i)] 
\item $F^p(x)$ is a singleton for each $x \in X$.  
\item $m=\sharp F(X)$.
\end{enumerate}
\end{lemma}
\begin{proof}
Using Lemma \ref{L:34} ``(ii)",
\[
\sharp F(X) = \sum^m_{i=1} \sharp F_{\bigcup^{i-1}_{j=1}W_j}(W_i) = \sum^m_{i=1} \sharp F^p(W_i)
\]
and this shows the statement.
\end{proof}

\begin{theorem} \label{T:unicity2.1}
The following statements are equivalent:
\begin{enumerate}[(i)] 
\item $F$ admits a unique alldifferent selection.  
\item $F^*(x)$ is a singleton for each $x \in X$.  
\item $F$ admits a Hall partition and $F^p(x)$ is a singleton for each $x \in X$.  
\item $F$ admits a Hall partition $(W_1, \ldots, W_m)$, $m \in \mathbb{N}$, such that $m=\sharp F(X)$.
\end{enumerate}
\end{theorem}
\begin{proof}
The equivalence of ``(i)" and ``(ii)" is straightforward. The equivalence of ``(ii)" and ``(iii)" had been 
proved in Theorems \ref{T:31} and \ref{T:41}. The equivalence of ``(iii)" and ``(iv)" is 
Lemma \ref{L:unicity2.2}.
\end{proof}

Under the assumption of the statements in Theorem \ref{T:unicity2.1} we obtain a statement on 
$\sharp X$.

\begin{lemma} \label{L:unicity2.3}
Let $(W_1, \ldots, W_m)$, $m \in \mathbb{N}$, be a Hall partition of $F$ such that $F^p(x)$ is a 
singleton for $x \in X$. $m=\sharp X$.
\end{lemma}
\begin{proof}
Suppose $m < \sharp X$. There exists $1 \le i \le m$ such that $\sharp W_i > 1$. Choose $x \in W_i$.
$\{x\} \ne W_i$ is a critical set of $F^p_{\mid W_i}=(F_{\bigcup^{i-1}_{j=1}W_j})_{\mid W_i}$.
This contradicts the definition of Hall partitions, since $W_i$ is a non-reducible set.
\end{proof}

\begin{lemma} \label{L:unicity2.4}
Let $(W_1, \ldots, W_m)$, $m \in \mathbb{N}$, be a Hall partition of $F$ such that $\sharp W_i = 1$ 
for $i=1, \ldots, m$. $m=\sharp X$.
\end{lemma}
\begin{proof}
We use the equation $\sharp X = \sum^m_{i=1} \sharp W_i = m$.
\end{proof}

\begin{lemma} \label{L:unicity2.5}
Let $(W_1, \ldots, W_m)$, $m \in \mathbb{N}$, be a Hall partition of $F$ such that $m=\sharp F(X)$.
$m=\sharp X$.
\end{lemma}
\begin{proof}
Using Lemma \ref{L:34} ``(ii)" and Lemma \ref{L:26}
\[
m 
= \sharp F(X) 
= \sum^m_{i=1} \sharp F_{\bigcup^{i-1}_{j=1}W_j}(W_i) 
\ge \sum^m_{i=1} \sharp W_i 
\ge m.
\]
This implies $\sharp W_i = 1$ for $i=1, \ldots, m$ and we apply Lemma \ref{L:unicity2.4}.
\end{proof}

                                        %
                                        %
\section{A Calculation Method} \label{S:algorithm}
We describe a method which calculates in finitely many steps the Hall partition of $F$. According 
to Theorem \ref{T:41} the Hall partition coincides with the alldifferent kernel, i.e., the alldifferent 
selections of $F$.

This type of algorithm had been described already by Easterfield \cite{Eas}. He proposed an 
algorithm, which finds the minimal sum over $n$ entries in an $n \times n$-matrix with real 
coefficients, where no two entries belong to the same row or the same column of the matrix.

Another algorithm had been proposed by R\'{e}gin \cite{Reg}. He starts with the calculation
of an alldifferent selection of $F$ and uses a necessary condition of Berge.

Description of the calculation method.

\begin{enumerate}[\qquad 1.]
\item 
Start with $F = F_{\emptyset}$ and $m=0$.
\medskip
\item
There exist an integer $m \ge 0$ and disjoint sets $W_1, \ldots, W_m$ in $X$ such that $W_i$ 
is a critical set of $F_{\bigcup^{i-1}_{j=1} W_j}$ for $i=1, \ldots, m$.
\medskip
\item
Choose $x \in X \backslash \bigcup^m_{j=1} W_j$ and define $W=\{x\}$.
\medskip
\item 
If $W$ satisfies $\sharp F_{\bigcup^{m}_{j=1} W_j}(W) < \sharp W$, STOP.   \\
Otherwise continue.
\medskip
\item 
If $W$ is a critical set of $F_{\bigcup^{m}_{j=1} W_j}$, define $W_{m+1} = W$,
set $m=m+1$ and continue with Step 8. Otherwise continue.
\medskip
\item 
If not yet all subsets $W^\prime$ of $X \backslash \bigcup^m_{j=1} W_j$ had been considered,
choose a new subset. If there exists a not yet considered subset 
$W^\prime \subset X \backslash \bigcup^m_{j=1} W_j$ 
of the same size like $W$, choose this $W^\prime$. Otherwise choose a set $W^\prime$ with one 
more element. Set $W=W^\prime$ and continue with Step 4.
\medskip
\item 
All subsets $W$ of $X \backslash \bigcup^m_{j=1} W_j$ had been considered. There is no critical set 
$W$ of $F_{\bigcup^{m}_{j=1} W_j}$. Define $W_{m+1} = X \backslash \bigcup_{j=1}^{m}W_j$,
set $m=m+1$ and STOP.
\medskip
\item 
If $\bigcup^{m}_{j=1} W_j = X$, STOP. Otherwise continue with Step 2.
\medskip
\end{enumerate}

The present calculation method consists of two loops. The inner loop from steps 4 - 6 determines a new 
critical set and the outer loop from steps 2 - 8 determines a Hall partition.

By definition of the calculation method it stops at one of the steps 4, 7 or 8. Therefore, after termination, 
$m \in \mathbb{N}$. If the calculation method does not stop at Step 4, it stops at Step 7 or Step 8 and 
terminates with sets $W_i \subset X$ for $i=1, \ldots, m$.

\begin{lemma} \label{L:CM.stops}
Let the calculation method not stop at Step 4.  \\
$(W_1, \ldots, W_m)$ defines a Hall partition of $F$.
\end{lemma}
\begin{proof}
By definition of the calculation method, $W_1, \ldots, W_m$ are disjoint (see Step 6 where 
$W^\prime \subset X \backslash \bigcup^m_{j=1} W_j$) and the union equals $X$ (see Step 7 
and Step 8). Condition ``(i)" of Definition \ref{D:partition} is satisfied by Step 4 (with singleton-sets
$W$). Condition ``(ii)" is satisfied, since $W$ is chosen with increasing size in Step 6 and 
Condition ``(iii)" is satisfied by Step 5. This shows $(W_1, \ldots, W_m)$ defines a Hall partition of $F$.
\end{proof}

\begin{lemma} \label{L:CM.Step4}
Let $F$ satisfy the Hall condition. The calculation method does not stop at Step 4.
\end{lemma}
\begin{proof}
Let $W_1, \ldots, W_m$ as defined by the calculation method and assume we are in Step 4. We define 
$W_{m+1} = X \backslash \bigcup_{i=1}^{m}W_i$. By definition of the calculation method
$W \subset W_{m+1}$ and $W_i$ is a critical set of $F_{\bigcup_{j=1}^{i-1} W_j}$ for $i=1, \ldots, m$. 
We apply Lemma \ref{L:32} ``(ii)" (with $m+1 \rightarrow m$) and obtain
\[
\sharp W
= \sharp W + \sum_{i=1}^{m}\sharp W_i - \sum_{i=1}^{m}\sharp F_{\bigcup_{j=1}^{i-1} W_j}(W_i)
\]
\[
= \sharp (W \cup \bigcup_{i=1}^{m}W_i) + \sharp F_{\bigcup_{j=1}^{m} W_j}(W) - \sharp F(W \cup \bigcup_{i=1}^{m}W_i)
\le \sharp F_{\bigcup_{j=1}^{m} W_j}(W).
\]
The calculation method does not stop at Step 4.
\end{proof}

\begin{theorem} \label{T:81}
The following statements are equivalent: 
\begin{enumerate}[(i)]
\item $F$ satisfies the Hall condition.   
\item The calculation method does not stop at Step 4.  
\item The calculation method stops at Step 7 or at Step 8 with a Hall partition of $F$. 
\end{enumerate}
\end{theorem}
\begin{proof}
``(i) $\Rightarrow$ (ii)" This implication had been proved in Lemma \ref{L:CM.Step4}.  \\
``(ii) $\Rightarrow$ (iii)" Assume the calculation method does not stop at Step 4. Then the calculation 
method stops at Step 7 or at Step 8. There is no other alternative. By Lemma \ref{L:CM.stops} the 
calculation method determines a Hall partition of $F$. \\
``(iii) $\Rightarrow$ (i)" If the calculation method determines a Hall partition, then $F$ satisfies the Hall 
condition by Theorem \ref{T:32}.
\end{proof}

The exits in Step 7 and Step 8 are distinguish by $\sharp F(X)$.

\begin{lemma} \label{L:CM.Step7}
Let the calculation method stop at Step 7. $\sharp F(X) > \sharp X$.
\end{lemma}
\begin{proof}
$(W_1, \ldots, W_m)$, $m \in \mathbb{N}$, is a Hall partition of $F$ by Lemma \ref{L:CM.stops} and 
$W_m$ is not a critical set of $F_{\bigcup^{m-1}_{j=1}W_j}$ (see Step 7). Using Lemma 
\ref{L:35} ``(ii)" and Theorem \ref{T:32}, $\sharp F(X) > \sharp X$.
\end{proof}

\begin{lemma} \label{L:CM.Step8}
Let the calculation method stop at Step 8. $\sharp F(X) = \sharp X$.
\end{lemma}
\begin{proof}
$(W_1, \ldots, W_m)$, $m \in \mathbb{N}$, is a Hall partition of $F$ by Lemma \ref{L:CM.stops} and 
$W_m$ is a critical set of $F_{\bigcup^{m-1}_{j=1}W_j}$(see Step 8). Using Lemma \ref{L:35} ``(ii)", 
$\sharp F(X) = \sharp X$.
\end{proof}

The present calculation method calculates the Hall partition of $F$, i.e., the alldifferent selections of $F$. 
This result cannot be achieved with the famous algorithm of Hopcroft and Karp (see \cite{HK}) which 
determines a single alldifferent selection of $F$.

On the other side it is possible to extend the present calculation method and calculate a single alldifferent
selection. Once we know the Hall partition $(W_1, \ldots, W_m)$ we select points $x_i \in W_i$ and 
$y_i \in F_{\bigcup^{i-1}_{j=1}W_j}(x_i)$ for $i=1, \ldots, m$ and apply the calculation method to each
\[
{\left({(F_{\bigcup^{i-1}_{j=1}W_j})}_{\mid W_i}\right)}_{\{x_i\},\{y_i\}}, 1 \le i \le m.
\]
A consecutive use of this process determines an alldifferent selection of $F$. Moreover, we have 
some freedom in the selection of the points $x_i$ and $y_i$ which can be used to satisfy 
additional restrictions.

Under some circumstances we may have additional information on the size of the images $F(x)$. 
This can be used to simplify the search of a critical set $W$ in step 4 - 6 in the calculation method. 
It is possible to neglect sets $W$ with $\sharp W < min\{ \sharp F(x) \mid x \in X \}$.

\begin{lemma} \label{L:85}
Let $W \subset X$ be a critical set of $F$. 
\[
\sharp W \ge min\{ \sharp F(x) \mid x \in X \}.
\]
\end{lemma}
\begin{proof}
By definition of critical sets, $\sharp W = \sharp F(W) \ge \sharp F(x)$
for each $x \in W$. This implies 
\begin{align*}
\sharp W 
& \ge max\{ \sharp F(x) \mid x \in W \}  \\
& \ge min\{ \sharp F(x) \mid x \in W\}    \\
& \ge min\{ \sharp F(x) \mid x \in X\}
\end{align*}
and proves the statement.
\end{proof}

                                        %
                                        %
\section{Application to Sudoku} \label{S:Sudoku}
A Sudoku is a square consisting of a 9$\times$9 grid which is
partly pre-populated by numbers between 1 and 9 called the givens.
The problem consists of finding numbers between 1 and 9 for all
unpopulated cells, such that each row, each column and each block 
consists of exactly the numbers $1, \ldots, 9$. The blocks of a Sudoku 
partition the Sudoku square into subsquares of size 3$\times$3. Each 
Sudoku consists of 9 rows, 9 columns and 9 blocks.

We set up a mathematical model for Sudoku puzzles. Let 
\[
G=\{(c_1, c_2) \mid c_1, c_2 \mbox{ are positive integers, }1 \le c_1 \le 9 \mbox{ and } 1 \le c_2 \le 9\}
\]
be the grid, i.e., the set of all cells. This set can be partitioned into the sets $G_p$ of populated cells 
and $G_u$ of unpopulated cells. For each populated cell $c \in G_p$ the given in cell $c$ is denoted by 
$g_c \in \{1, \ldots, 9\}$.

Let $row(c)$, $column(c)$ and $block(c)$ denote the row, column and block containing a 
cell $c\in G$. Define
\[
n(c)=row(c) \cup column(c) \cup block(c),
\]
i.e., $n(c)$ describes the union of the row, column and block, which contain the cell $c \in G$. 

For each populated cell $c \in G_p$ we define $F(c)=\{g_c\}$. For each unpopulated cell $c \in G_u$, 
let $F(c)$ be the markup of $c$, i.e., 
\[
F(c) = \{ y \in \{1, \ldots, 9\} \mid y \ne g_d\mbox{ for each cell }d \in n(c) \cap G_p \}.
\]
Let $C$ be any row, column or block of the Sudoku grid. We restrict ourselves to the unpopulated cells
in $C$, i.e., we set $X = C \cap G_u$, set $Y=\bigcup_{c \in X}F(c)$ and consider the set-valued 
mapping $F_{\mid X}: X \longrightarrow 2^Y$. Then $X$ and $Y$ are finite nonempty sets and we 
can apply the calculation method of Section \ref{S:algorithm} to $F_{\mid X}$.

\begin{figure}
\includegraphics{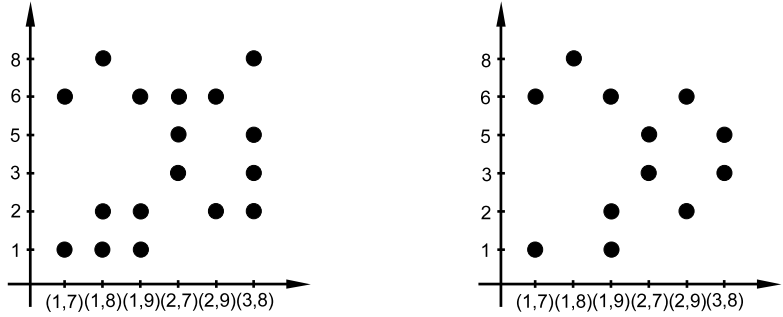}
\caption{Shortz' Example 301 with the calculation method.} 
\label{Fig5}
\end{figure}

We consider a continuation of Crook \cite[Fig. 7]{Cro} of Example 301 of Shortz \cite{Sho}
and choose $C$ as the upper right block. Then 
\[
X=\{(1, 7), (1, 8), (1, 9), (2, 7), (2, 9), (3, 8)\}
\]
and $Y=\{1, 2, 3, 5, 6, 8\}$. The mapping $F$ is depicted in Fig. \ref{Fig5} (left) and the 
application of the calculation method in Section \ref{S:algorithm} results in the following steps.

The set $W_1=\{(1, 7), (1, 9), (2, 9)\}$ describes a critical and non-reducible set of $F$ with
$F(W_1)=\{1, 2, 6\}$. The set $W_2=\{(1, 8)\} \subset X \backslash W_1$ describes
a critical and non-reducible set of $F_{W_1}$ with $F_{W_1}(W_2)=\{8\}$. Finally, the set
$W_3 = \{(2, 7), (3, 8)\} \subset (X \backslash W_1) \backslash W_2$ describes a critical and 
non-reducible set of $(F_{W_1})_{W_2} = F_{W_1 \cup W_2}$ (see Lemma \ref{L:21}) with 
$F_{W_1 \cup W_2}(W_3)=\{3, 5\}$. The Hall partition of $F$ consists of $(W_1,W_2,W_3)$ and 
is depicted on the right of Fig. \ref{Fig5}.

Knowing all alldifferent selections of $F$ we see that $F^*((1, 8))$ consists of the single value $8$, i.e.,
the cell $(1, 8)$ contains the value $8$.

The calculation method from Section \ref{S:algorithm} provides the same result like the strategies proposed 
by Crook \cite{Cro} using preemptive sets and Provan \cite{Pro} using a pigeon-hole rule. This calculation 
method provides a completion and theoretical foundation of these strategies.

                                        %
                                        %

\end{document}